\documentclass[10pt,a4paper,oneside]{amsart}

\usepackage{amsfonts, amsmath, amssymb, amsthm, hyperref}
\usepackage{anysize,paralist}
\usepackage[usenames,dvipsnames]{xcolor}

\newtheorem{theorem}{Theorem}[section]
\newtheorem*{theorem*}{Theorem}
\newtheorem{lemma}[theorem]{Lemma}

\newtheorem{corollary}[theorem]{Corollary}

\theoremstyle{definition}
\newtheorem{definition}[theorem]{Definition}

\theoremstyle{remark}
\newtheorem{remark}[theorem]{Remark}

\numberwithin{equation}{section}

\newcommand{\pt}{\mathrm{pt}}
\newcommand\Sym{\mathfrak S}

\newcommand{\sign}{\mathrm{sign}}

\newcommand{\E}{\mathrm{E}}

\newcommand{\R}{\mathbb{R}}
\newcommand{\Z}{\mathbb{Z}}
\newcommand{\F}{\mathbb{F}}

\DeclareMathOperator{\vol}{vol}

\DeclareMathOperator{\cat}{cat}

\DeclareMathOperator{\const}{const}

\renewcommand{\int}{\mathop{\rm int}}

\renewcommand{\epsilon}{\varepsilon}

\begin{document}

\title[The Schwarz genus of the Stiefel manifold and counting geometric configurations]{The Schwarz genus of the Stiefel manifold and counting geometric configurations}

\author{Pavle~Blagojevi\'c}
\email{pavle.v.m.blagojevic@gmail.com}
\address{Pavle~Blagojevi\'c, Institut f\" ur Mathematik, FU Berlin, Arnimallee 2, 14195 Berlin, Germany; and Mathemati\v cki Institut SANU, Knez Michailova 35/1, 11001 Beograd, Serbia}
\thanks{The research of Pavle Blagojevi\'c leading to these results has received funding from the European Research Council under the European Union's Seventh Framework Programme (FP7/2007-2013) / ERC Grant agreement no.~247029-SDModels. He is also supported by the grant ON 174008 of the Serbian Ministry of Science and Environment.}

\author{Roman~Karasev}
\thanks{The research of R.N.~Karasev is supported by the Dynasty Foundation, the President's of Russian Federation grant MD-352.2012.1, the Russian Foundation for Basic Research grants 10-01-00096 and 10-01-00139, the Federal Program ``Scientific and scientific-pedagogical staff of innovative Russia'' 2009--2013, and the Russian government project 11.G34.31.0053.}
\email{r\_n\_karasev@mail.ru}
\address{Roman~Karasev, Department of Mathematics, Moscow Institute of Physics and Technology, Institutskiy per. 9, Dolgoprudny, Russia 141700; and Laboratory of Discrete and Computational Geometry, Yaroslavl' State University, Sovetskaya st. 14, Yaroslavl', Russia 150000}

\begin{abstract}
In this paper we compute: the Schwarz genus of the Stiefel manifold $V_k(\R^n)$ with respect to the action of the Weyl group $W_k:=(\Z/2)^{k}\rtimes\Sym_k$,
and the Lusternik--Schnirelmann category of the quotient space $V_k(\R^n)/W_k$.
Furthermore, these results are used in estimating the number of: critically outscribed parallelotopes around the strictly convex body,
and Birkhoff--James orthogonal bases of the normed finite dimensional vector space.
\end{abstract}

\maketitle

\begin{center}
 {\em Dedicated to Professor Ljubi\v sa Ko\v cinac on the occasion of his 65th birthday}
\end{center}

\section{Introduction}

The classical problem of estimating the number of the periodic billiard trajectories in a strictly convex body, introduced by G. Birkhoff in early 1990s, was recently studied in series of papers
by  Farber \cite{farber}, Farber \& Tabachnikov \cite{FarTab}, and Karasev \cite{kar2009billiard} via topological methods.
The problem of counting the number of periodic billiard trajectories was connected with the problem of determining the critical points of the appropriately defined length function on a cyclic
configuration space.
Then the number of critical points was estimated from below by the Lusternik--Schnirelmann category of the quotient of the cyclic configuration space.

In this paper we study the Schwarz genus of the Stiefel manifold $V_k(\R^n)$ with respect to the action of the Weyl group $W_k:=(\Z/2)^{k}\rtimes\Sym_k$ and
consequently the Lusternik--Schnirelmann category of the quotient manifold $V_k(\R^n)/W_k$, Section \ref{section:schwarzgenus}.
In particular, in Corollary \ref{corollary:st-ls} we prove that:
\[
g_{W_k}(V_k(\R^n)) = \cat V_k(\R^n)/W_k = nk - \tfrac{k(k+1)}{2} + 1.
\]
Following the spirit of arguing as in the billiard problem we use the result on the Lusternik--Schnirelmann category of the quotient manifold $V_k(\R^n)/W_k$ to estimate the number of:
critically outscribed parallelotopes and Birkhoff--James orthonormal bases, Section \ref{section:counting}.
We prove that:
{\em
\begin{compactitem}[$\bullet$]
 \item Every strictly convex body $K\subset \mathbb R^n$ has at least $\tfrac{n(n-1)}{2}+1$ distinct critically outscribed parallelotopes;
 \item Every smooth norm on  $\R^n$ has at least $\tfrac{n(n-1)}{2}+1$ Birkhoff--James orthonormal bases.
\end{compactitem}
}

\smallskip

\noindent{\bf Acknowledgments.} We are grateful to Aleksandra Dimitrijevi\'c Blagojevi\'c for careful reading of the manuscript and valuable suggestions.

\section{The Schwarz genus of the Stiefel manifold}
\label{section:schwarzgenus}

In this section we evaluate the Schwarz genus of the Stiefel manifold $V_k(\R^n)$ of orthonormal $k$-frames in $\R^n$ with respect to the action of the group $W_k:=(\Z/2)^{k}\rtimes\Sym_k$ and prove that
\[
 g_{W_k}(V_k(\R^n)) = \dim V_k(\R^n) + 1 = nk - \frac{k(k+1)}{2} + 1 .
\]
Moreover, as a consequence we obtain the following estimate for the Lusternik--Schnirelmann category of the unordered configuration space of the projective space
\[
 \cat B_k(\mathbb RP^n) \ge nk - \frac{k(k-1)}{2} + 1.
\]

\subsection{The Schwarz genus}
Following the original paper by Schwarz and adapting the notion of the genus of a regular fibration~\cite[Chapter V]{schw1966} we introduce the Schwarz genus for a free $G$-space.
For more details and applications consult \cite{bart1993} and \cite{vol2000}.

\begin{definition}
The \emph{Schwarz genus} $g_G(X)$ of a free $G$-space $X$ is the smallest number $n$ such that $X$ can be covered by $n$ open $G$-invariant subsets $X_1,\ldots, X_n$
with the property that for every $1\leq i\leq n$ there exists a $G$-equivariant map $X_i\to G$.
The Schwarz genus of a free $G$-space $X$ will be denoted by $g_G(X)$.
\end{definition}

For a paracompact $G$-space $X$ the Schwarz genus coincides with the smallest number $n$ such that there exists a $G$-equivarian map $X\to G^{*n}$
Here $G^{*n}$ denotes the $n$-fold join of the group $G$ with diagonal $G$-action. 
The group $G$ is considered as a $0$-dimensional simplicial complex.

Let $\E_rG$ denote a compact free $r$-dimensional $G$-space that is also $(r-1)$-connected.
The following natural generalization of the Borsuk--Ulam theorem
\begin{center}
\emph{There is no $G$-equivariant map $\E_rG\to\E_{r-1}G$,} 
\end{center}
implies the $g_G(\E_rG)=r+1$.

\medskip

The classical notion of Lusternik-Schnirelmann category of a topological space can be defined as follows.
\begin{definition}
The \emph{Lusternik--Schnirelmann category} $\cat X$ of a space $X$ is the smallest integer $n$ for which $X$ can be covered by $n$ open subsets $X_1, X_2,\ldots, X_{n}$
such that the inclusions $X_i \to X$ are nullhomotopic.
\end{definition}

The Schwarz genus of a free $G$-space and the Lusternik--Schnirelmann category of its quotient $X/G$ are connected via the following lemma.

\begin{lemma}
\label{lemma:ls-by-gen}
For a free $G$-space $X$ the following inequality holds
$$
\cat X/G \ge g_G(X).
$$
\end{lemma}
\begin{proof}
Let $m = \cat X/G$ and $Y_1\cup \dots \cup Y_m$ be the corresponding open covering of $X/G$ with null-homotopic inclusions $Y_i\subseteq X/G$. 
Then the maps $X_i = \pi^{-1}(Y_i)\to Y_i$ defined as restrictions of the natural projection $\pi : X \to X/G$ are trivial coverings.
This means that there exist $G$-equivariant homeomorphisms $X_i\to Y_i\times G$ where the actions on $Y_i\times G$ are given by $g\cdot (y,h):=(y,gh)$.
Composing these homeomorphisms with the projection on the second factor we get $G$-equivariant maps $X_i\to G$.
Thus, by the definition of the Schwarz genus $g_G(X)\le m$.
\end{proof}

Use of the Schwarz genus is sometimes more convenient, compared to the Lusternik--Schnirelmann category, because the Schwarz genus is monotone with respect to inclusions of $G$-invariant spaces.
This property will be essentially used in the proofs of Corollaries \ref{corollary:proj-conf-ls}, \ref{corollary:outscr-paral} and \ref{corollary:bj-bases}.

One of the main tools to estimate the Schwarz genus and consequently the Lusternik--Schnirelmann category of the quotient is the following cohomological criterion that is a particular case of \cite[Theorem 12, page 91]{schw1966}.
\begin{lemma}
\label{lemma:gen-by-coh}
Let $G$ be a finite group, $R$ be a commutative ring with unit, $M$ be an $R[G]$-module and $n>0$ be an integer.
If the map in the equivariant cohomology $\pi^*:H^n_G(\pt;M)\to H_G^n(X;M)$ is nonzero, then
$$
g_G(X)\ge n+1.
$$
The map $\pi^*$ is induced by the $G$-equivariant projection $\pi:X\to\pt$.
\end{lemma}

\subsection{The Schwarz genus of the Stiefel manifold of orthonormal $k$-frames in $\R^d$}
\label{subsection:2.2}

Let $V_k(\R^n)$ be the Stiefel manifold of all orthonormal $k$-frames in $\R^n$, and $G\subseteq \mathrm{O}(k)$.
Then $G$ acts naturally on the Stiefel manifold $k$-frames by
\[
(v_{1},\ldots ,v_{k})\cdot g=\Big( \sum_{j=1}^k v_{j}\gamma_{j1},\ldots,\sum_{j=1}^k v_{j}\gamma_{jk}\Big).
\]
Here $(v_{1},\ldots ,v_{k})\in V_k(\R^n)$ and $g=\big(\gamma_{ij}\big)_{i,j=1}^{k}\in \mathrm{O}(k)$.
The action is from the right, but it transforms into a left action by $g\cdot (v_{1},\ldots ,v_{k}):=(v_{1},\ldots ,v_{k})\cdot g^{-1}$.

We consider the action of the group $W_k:=(\Z/2)^{k}\rtimes\Sym_k\subset \mathrm{O}(k)$ on the Stiefel manifold $V_k(\R^n)$.
Let $\varepsilon_1,\ldots ,\varepsilon_n$ be generators of the component $(\Z/2)^{n}$ and $\pi\in\Sym_k$.
Then for $(v_1,\ldots ,v_k) \in V_k(\R^n)$ we set
\[
\varepsilon_i\cdot (v_{1},\ldots ,v_k) = (v_1^{\prime },\ldots ,v_k^{\prime }),
\]
where $v_i^{\prime}=-v_i$ and $v_{j}^{\prime }=v_{j}$ for all $j\neq i$, and
\[
\pi \cdot (v_1,\ldots ,v_k) = ( v_{\pi (1)},\ldots ,v_{\pi (k)}).
\]

The space $M_{k}$ of all real $k\times k$-matrices can be considered as a real $\mathrm{O}(k)$-representation with respect to the matrix conjugation action.
Consequently, $M_{k}$ is also a real $W_k$-representation.
We consider the following real $W_k$-subrepresentations of $M_{k}$:
\begin{compactitem}[$\circ$]
\item $R_k$ the space of all symmetric $k\times k$ matrices with zeroes on the diagonal, and
\item $P_k$ the space of all $k\times k$ matrices with zeroes outside the diagonal.
\end{compactitem}
Notice that $\dim R_k=\tfrac{k(k-1)}{2}$ and $\dim P_k=k$.

Let $X$ be a free $G$-space, and $V$ be a $d$-dimensional real $G$-representation.
Then $\xi_{X,V}$ denotes the following flat vector bundle
\[
V\rightarrow X\times_{G} V \rightarrow X/G,
\]
and $w_{X,V}:=w_d(\xi_{X,V})\in H^{d}(X/G;\F_2)\cong H^{d}_G(X;\F_2)$ its top Stiefel--Whitney class.

\begin{theorem}
\label{theorem:st-top-coh}
Let $n\geq 3$, $1\leq k\leq n$, $d:=\dim V_k(\R^n)= nk - \tfrac{k(k+1)}{2}$ and $U:=R_k\oplus P_k^{\oplus (n-k)}$.
The map in equivariant cohomology induced by the natural projection $\pi:X\to\pt$:
\[
\pi^*:H^{d}_{W_k}(\pt;\F_2)\to H^{d}_{W_k}(V_k(\R^n);\F_2)
\]
is non-zero.
In particular,
\[
 w_{V_k(\R^n),U}\in H^{d}(V_k(\R^n)/W_k;\F_2)\cong H^{d}_{W_k}(V_k(\R^n);\F_2)
\]
does not belong to the kernel of $\pi^*$, i.e., is the generator of the group $H^{d}(V_k(\R^n)/W_k;\F_2)$.
\end{theorem}
\begin{proof}
We prove the theorem in the following steps:
\begin{compactitem}[$\circ$]
\item we first equip the Stiefel manifold $V_k(\R^n)$ with the free $W_k$-CW-structure, then

\item identify the Stiefel--Whitney class $w_{V_k(\R^n),U}$ with the equivariant {\rm (mod 2)}-primary obstruction
\[
 \gamma_{\Z/2}^{W_k}(V_k(\R^n),S(U))\in H^{d}_{W_k}(V_k(\R^n);\F_2),
\]
and finally

\item prove that the equivariant primary obstruction which decides the existence of $W_k$-equivariant map $V_k(\R^n)\to S(U)$:
\[
 \gamma^{W_k}(V_k(\R^n),S(U))\in H^{d}_{W_k}(V_k(\R^n);\mathcal{Z}),
\]
being also the preimage of $\gamma_{\Z/2}^{W_k}(V_k(\R^n),S(U))$ under the coefficient reduction $W_k$-morphism $\mathcal{Z}\to\F_2$, is the generator of the group
$H^{d}_{W_k}(V_k(\R^n);\mathcal{Z})\cong\Z/2$. Moreover,
\item our proof will yield that coefficient reduction $W_k$-morphism $\mathcal{Z}\to\F_2$ induces an isomorphism of the groups
\[
 H^{d}_{W_k}(V_k(\R^n);\F_2)\cong H^{d}_{W_k}(V_k(\R^n);\mathcal{Z}).
\]

\end{compactitem}
Here $\mathcal{Z}$ denotes the homotopy group $\pi_{d-1}(S(U))$ of the sphere $S(U)$ considered as a $W_k$-module.
In particular, it is important to deduce that for any generator $\varepsilon_i$ of the subgroup $(\Z/2)^k$ of $W_k$ and $z\in\mathcal{Z}$ the $W_k$-action is described by $\varepsilon_i\cdot z=(-1)^{n-1}z$.

\smallskip

\noindent
\textbf{(1)}
The Stiefel manifold $V_k(\R^n)$ is a $d$-dimension smooth closed $W_k$-manifold.
The smooth structure on $V_k(\R^n)$ can be assumed to be invariant with respect to the action of $W_k$.
Thus, the quotient manifold $V_k(\R^n)/W_k$ is a $d$-dimensional smooth closed manifold.
Therefore, $V_k(\R^n)/W_k$ can be triangulated \cite[Theorem, page 389]{cairns}.
Moreover, for a given point $x\in V_k(\R^n)/W_k$ the triangulation can be chosen in such a way that $x$ belongs to an interior of some maximal cell.
Since the action of the group $W_K$ on $V_k(\R^n)$ is free, the triangulation of the quotient $V_k(\R^n)/W_k$ transforms the Stiefel manifold $V_k(\R^n)$
into a free $W_k$-CW-complex.
The $W_k$-CW-structure on $V_k(\R^n)$ has an additional property that the preimages of $x$, with respect to the quotient map $V_k(\R^n)\to V_k(\R^n)/W_k$, belong to interiors of different maximal cells.

\smallskip

\noindent
\textbf{(2)}
Since $V_k(\R^n)$ is a $d$-dimensional connected free $W_k$-CW-complex and $\dim U=d$, then
\begin{compactitem}[-]
\item using \cite[Lemma 5.3]{BlLuZi2012} the Stiefel--Whitney class $w_{V_k(\R^n),U}$ can be identified with the equivariant {\rm (mod~2)}-primary obstruction
\[
 \gamma_{\Z/2}^{W_k}(V_k(\R^n),S(U))\in H^{d}_{W_k}(V_k(\R^n);\F_2)\cong H^{d}(V_k(\R^n)/W_k;\F_2),
\]
\item the quotient space $V_k(\R^n)/W_k$ is a compact manifold, and consequently
\[
 H^{d}(V_k(\R^n)/W_k;\F_2)\cong\F_2.
\]
\end{compactitem}
The equivariant {\rm (mod 2)}-primary obstruction $\gamma_{\Z/2}^{W_k}(V_k(\R^n),S(U))$ is the image of the equivariant primary obstruction $\gamma^{W_k}(V_k(\R^n),S(U))\in H^{d}_{W_k}(V_k(\R^n);\mathcal{Z})$
under the coefficient reduction $W_k$-morphism $\mathcal{Z}\to\F_2$, \cite[Section 5.1]{BlLuZi2012}.
The ambient group of the obstruction can be identified vie the Equivariant Poincar\'e duality, as explained in \cite[Theorem 1.4, page 2638]{BlaBla},
\[
H^{d}_{W_k}(V_k(\R^n);\mathcal{Z})\cong H_{0}^{W_k}(V_k(\R^n);\mathcal{Z}\otimes\mathcal{V})
\]
where $\mathcal{V}\cong_{\mathrm{Ab}}\Z$ is the orientation character of the Stiefel manifold $V_k(\R^n)$ with respect to the group $W_k$.
The action of $W_k$ on $\mathcal{V}$ is given by
\[
\varepsilon_i\cdot v=(-1)^n v\text{~~~~~,~~~~~}\pi\cdot v=(\sign(\pi))^n v
\]
for $v\in\mathcal{V}$, $\varepsilon_i$ a generator of the subgroup $(\Z/2)^k$ and $\pi\in\Sym_k\subset W_k$.
Here $\sign(\pi)\in\{1,-1\}$ denotes the sign of permutation $\pi$.

In the case when the Stiefel manifold $V_k(\R^n)$ is simply connected we have a sequence of isomorphisms, as in 
\cite[Section 1.5, page 2639]{BlaBla},
\[
 H_{0}^{W_k}(V_k(\R^n);\mathcal{Z}\otimes\mathcal{V})\cong H_{0}(W_k;\mathcal{Z}\otimes\mathcal{V})\cong (\mathcal{Z}\otimes\mathcal{V})_{W_k},
\]
where the last group is the group of $W_k$-coinvariants of the $W_k$-module $\mathcal{Z}\otimes\mathcal{V}$.
Since every element $\varepsilon_i$ of the group $W_k$ acts on the $W_k$-module $\mathcal{Z}\otimes\mathcal{V}$ by multiplication by ``$(-1)$''
we get that the group of coinvariants is isomorphic with $\Z/2$.
When the Stiefel manifold $V_k(\R^n)$ is not simply connected, for $k\in\{n-1,n\}$, direct calculation yields the same conclusion.

Thus, we have proved that
\[
 H^{d}_{W_k}(V_k(\R^n);\mathcal{Z})\cong\Z/2.
\]

\smallskip

\noindent
\textbf{(3)} Finally, we evaluate the equivariant primary obstruction $\gamma^{W_k}(V_k(\R^n),S(U))$ using the general position map scheme \cite[Definition 1.5, page 2639]{BlaBla} and already described
$W_k$-CW-structure on $V_k(\R^n)$ where the distinguished point $x\in V_k(\R^n)/W_k$ will be chosen later.

Let $\R^k\subset\R^n$ be the standard inclusion and $x_1,\ldots,x_n:\R^n\to\R$ denote coordinate functionals.
Consider a symmetric quadratic form $\phi\colon\R^n\times\R^n\to\R$ that has generic restriction $\phi|_{\R^k\times\R^k}$, i.e.,
the restriction form $\phi|_{\R^k\times\R^k}$ has pairwise distinct eigenvalues.
The equivariant map in general position $\tau=(\tau_0,\tau_1,\ldots ,\tau_{n-k})\colon V_k(\R^n)\to U=R_k\oplus P_k^{\oplus (n-k)}$ is defined by
\[
 \tau_0(v_1,\ldots,v_k)=[\phi(v_i,v_j)]_{1\leq i<j\leq k},
\]
and for $1\leq j\leq n-k$ by
\[
 \tau_j(v_1,\ldots,v_k)=(x_{k+j}(v_1),\ldots,x_{k+j}(v_k)).
\]

Now we are interested in the solution of the system of equations $\tau(v_1,\ldots,v_k)=0\in U$ modulo the group action $W_k$.
The subsystem of equations
\begin{eqnarray*}
\tau_1(v_1,\ldots,v_k)&=&(x_{k+1}(v_1),\ldots,x_{k+1}(v_k)) =0 \\
\tau_2(v_1,\ldots,v_k)&=&(x_{k+2}(v_1),\ldots,x_{k+2}(v_k)) =0 \\
\ldots                & & \ldots			      \\
\tau_{n-k}(v_1,\ldots,v_k)&=&(x_{n}(v_1),\ldots,x_{n}(v_k)) =0
\end{eqnarray*}
implies that the $k$-frame $v_1,\ldots,v_k$ belongs to $\R^k$.
The condition of generic restriction on $\phi$ with the remaining set of equations $\tau_0(v_1,\ldots,v_k)=0\in R_k$ implies that the restricted symmetric quadratic form $\phi|_{\R^k\times\R^k}$
on the solution frame $v_1,\ldots,v_k$ has to be diagonal.
Thus, a solution of the  system of equations $\tau(v_1,\ldots,v_k)=0\in U$  exists and solution frame $v_1,\ldots,v_k$ is unique up to the action of the group $W_k$.
Moreover, it can be verified that the solution is non-degenerated.

If the $W_k$-CW-structure on $V_k(\R^n)$ is chosen in such a way that the distinguished point is the solution frame $v_1,\ldots,v_k$, then the equivariant obstruction cocycle induced by the map $\tau$,
that lives in $C^{d}_{W_k}(V_k(\R^n);\mathcal{Z})$, evaluates non-trivially only on the cell that contains distinguished point.
More precisely, after suitable choice of orientations it evaluates to $1\in\mathcal{Z})$.
Thus, the equivariant primary obstruction $\gamma^{W_k}(V_k(\R^n),S(U))$ is the generator of the group $H^{d}_{W_k}(V_k(\R^n);\mathcal{Z})$ and therefore does not vanish.
This concludes the proof of the theorem.
\end{proof}

A direct consequence of the previous theorem is the following lower bound estimate for the Schwarz genus of the Stiefel manifold.

\begin{corollary}
\label{corollary:st-ls}
\[
g_{W_k}(V_k(\R^n)) = \cat V_k(\R^n)/W_k = \dim V_k(\R^n) + 1 = nk - \tfrac{k(k+1)}{2} + 1.
\] 
\end{corollary}

\begin{proof}
From Lemmas~\ref{lemma:gen-by-coh} and \ref{lemma:ls-by-gen} we get:
\[
\cat V_k(\R^n)/W_k \ge g_{W_k}(V_k(\R^n)) \ge nk - \tfrac{k(k+1)}{2} + 1.
\]
On the other hand the dimension bounds the Lusternik--Schnirelmann category from above \cite[Theorem 1.7, page 4]{CLOT} and thus
\[
\cat V_k(\R^n)/W_k \le \dim V_k(\R^n) + 1 = nk - \tfrac{k(k+1)}{2} + 1.
\]
Therefore, the equality holds in both cases.
\end{proof}

\subsection{The Lusternik--Schnirelmann category of the unordered configuration space of the projective space}

Let $G$ be a finite group that acts on the topological space $X$, and $k>1$ be an integer.
The $G$-\emph{ordered configuration space} of $k$ pairwise distinct points in $X$ is defined to be:
\[
 F_G(X,k):=\{ (x_1,\ldots, x_k)\in X^k:x_i\neq g\cdot x_j\text{~for~}i\neq j\text{~and~}g\in G\}.
\]
The symmetric group $\Sym_k$ acts naturally on the $G$-ordered configuration space by permuting the factors.
The quotient space $B_G(X,k):=F_G(X,k)/\Sym_k$ is called the $G$-\emph{unordered configuration space} of $k$ pairwise distinct points in $X$.
When $G$ is the trivial group, the $G$-ordered configuration space $F_G(X,k)$ coincides with the classical ordered configuration space $F(X,k)$; and similarly for unordered configuration space $B(X,k)$.

Using the results of the previous section we give a lower bound on the Lusternik--Schnirelmann category of the unordered configuration space of the projective space $B(\R P^d,k)$ when $d\geq 3$.

\begin{corollary}
\label{corollary:proj-conf-ls}
For $d\geq 3$:
$$
\cat B(\R P^d,k) \ge dk - \tfrac{k(k-1)}{2} + 1.
$$ 
\end{corollary}

\begin{proof}
Let $\Z/2$ act antipodaly on the sphere $S^{d}$.
The group $W_k$ acts naturally on the $\Z/2$-configuration space $F_{\Z/2}(S^{d},k)$.
Then there is a homeomorphism:
\[
 B(\R P^d,k)\approx B_{\Z/2}(S^d,k)/W_k.
\]
Since $V_k(\R^{d+1})\subset B_{\Z/2}(S^d,k)$, then the monotonicity of the Schwarz genus under equivariant inclusion (see~\cite{schw1966} or \cite{vol2000}) implies:
\[
g_{W_k}(B_{\Z/2}(S^d,k)) \ge g_{W_k}(V_k(\R^{d+1})) = (d+1)k - \frac{k(k+1)}{2} + 1 = dk - \frac{k(k-1)}{2} + 1.
\] 
Application of Lemma~\ref{lemma:ls-by-gen} concludes the proof.
\end{proof}

\begin{remark}
For $k=2$ the space $B_{\Z/2}(S^d,k)$ is $W_k$-equivariantly homotopy equivalent to $V_k(\R^{d+1})$ and therefore we have the equality:
\[
\cat B_2(\R P^d) = 2d.
\]
\end{remark}

\begin{remark}
Results of~\cite{kar2009conf}, with more detailed proofs in~\cite{BlLuZi2012}, give a lower bound for the Lusternik--Schnirelmann category of $B(M,k)$ in the form $(d-1)(k - D_p(k)) + 1$
for a $d$-dimensional manifold $M$ and $D_p(k)$ the sum of digits in the $p$-adic expansion of $k$ for the prime $p$.
Thus, for small $k$ the bound from Corollary~\ref{corollary:proj-conf-ls} improves this general bound.
\end{remark}

\section{Counting geometric configurations}
\label{section:counting}

In this section we use the results of Section \ref{section:schwarzgenus} in order to count different geometric configurations: critically outscribed parallelotopes and Birkhoff--James orthonormal bases.

\subsection{Critically outscribed parallelotopes}
\label{subsection:parallelotpes}
Now we are going to consider parallelotopes outscribed around a given convex body and estimate the possible number of critical ones.

A \emph{parallelotope} is a convex polytope $P$ in $\R^n$, bounded by $2n$ pairwise distinct hyperplanes
\[
 H_1^0, H_1^1, H_2^0, H_2^1,\ldots, H_n^0, H_n^1
\]
such that $H_i^0$ and $H_i^1$ are parallel for every $i$, and $H_1^0,H_2^0,\ldots,H_n^0$ are in general position.
The class of parallelotopes naturally extends the class of parallelograms from the plane to higher dimensions.

\begin{definition}
\label{definition:critical}
A parallelotope $P\subset \R^n$ is \emph{critically outscribed} around the convex body $K$ if:
\begin{compactenum}
\item $K\subseteq P$, and
\item for every pair of its parallel defining hyperplanes $H_i^0$ and $H_i^1$ there exist points $x_0\in H_i^0\cap K$ and $x_1\in H_i^1\cap K$ such that the line determined by $x_0$ and $x_1$ is
parallel with all the remaining hyperplanes $H_j^0$, $H_j^1$ where $j\neq i$.
\end{compactenum}
\end{definition}

The notion of a critically outscribed parallelotope arises naturally in the context of finding parallelotopes of minimal volume containing convex body $K$.
Every parallelotope containing $K$ that has the minimal volume must be critically outscribed.
Thus, it is important to estimate the possible number of critically outscribed parallelotope around a convex body.

\begin{corollary}
\label{corollary:outscr-paral}
Every strictly convex body $K\subset \mathbb R^n$ has at least $\tfrac{n(n-1)}{2}+1$ distinct critically outscribed parallelotopes.
\end{corollary}

\begin{proof}
Consider the manifold $M$ of all $n$-frames $(e_1,\ldots, e_n)$, with unit vectors $e_i$ not necessarily orthogonal to each other.
The manifold $M$ contains $V_n(\R^n)$ as a $W_n$-equivariant subspace.
Therefore by Corollary~\ref{corollary:st-ls} and the monotonicity of the Schwarz genus we have that
\[
\cat M/W_n \ge g_{W_n} (M) \ge g_{W_n} (V_n(\R^n)) = \tfrac{n(n-1)}{2}+1.
\]

Consider a function $f\colon M\to\R$ defined as follows.
For any $n$-frame $(e_1,\ldots, e_n)\in M$, find the support hyperplanes $H_i^0, H_i^1$ to $K$ orthogonal to $e_i$ such that $\langle e_i, H_i^1\rangle > \langle e_i, H_i^0\rangle$.
Here $\langle\cdot,\cdot\rangle$ denotes the standard inner product of $\R^n$
These hyperplanes bound a parallelotope $P(e_1,\ldots,e_n)$.
Put
\[
f(e_1, \ldots, e_n) := \vol P(e_1,\ldots,e_n).
\]
The function $f\colon M\to\R$ is a smooth $W_n$-invariant proper function on $M$.
Here proper function means that the preimage of every compact set is also compact.
Thus, $f$ has at least $\cat M/W_n$ critical $W_n$-orbits by the standard Lusternik--Schnirelmann theory, \cite[Theorem 1.15, page 7]{CLOT}.

Finally, we note that distinct critical orbits of the function $f\colon M\to\R$ correspond to the distinct critically outscribed parallelotopes around $K$.
Indeed, consider the dependence of $\vol P(e_1,\ldots, e_n)$ on some $e_i$.
Without loss of generality let it be $e_1$.
The support hyperplanes to $K$ with fixed normals $\pm e_2,\ldots, \pm e_n$ form a cylinder $Z$.
Now varying supporting hyperplanes $H_1^0$ and $H_1^1$ with normals $-e_1$ and $e_1$ cut the parallelotope $P(e_1,\ldots, e_n)$ out of the cylinder $Z$.
Let $\ell$ be the unique line orthogonal to $e_2,\ldots,e_n$.
If we only vary $e_1$ then the volume $\vol P(e_1,\ldots, e_n)$ is proportional to the length $L(e_1)$ of the segment cut on $\ell$ by the hyperplanes $H_1^0$ and $H_1^1$.
Parameterizing $\ell$ by $t\cdot v$ for some unit vector $v$ and $t\in\R$ and introducing the support function
\[
s(K, y) = \sup_{x\in K} \langle y, x\rangle
\]
we obtain the expression
\[
L(e_1) = \frac{s(K, e_1)}{\langle v, e_1\rangle} - \frac{s(K, -e_1)}{\langle v, -e_1\rangle}.
\]
This expression does not depend on the length of $e_1$.
Thus we can normalize $e_1$ by $\langle e_1, v \rangle = 1$ instead of $\|e_1\| = 1$.
Under this constraint the expression simplifies to $L(e_1) = s(K, e_1) + s(K, -e_1)$.
Recall that the directional differential of the support function can be identified in the case of strict convexity with the support point.
For $s(K, e_1)$ this is $p_+ = H_1^0+\cap K$ and for $s(K, -e_1)$ this is $p_- = H_1^1-\cap K$.
Hence
\[
dL(e_1) = p_+ - p_-.
\]
Recalling the constraint $\langle v, e_1\rangle=1$ we see that the condition for a critical point is that $p_+ - p_-\parallel \ell$.
Applied for each base vector $e_i$ this gives the definition of a critically outscribed parallelotope.
\end{proof}

\begin{remark}
Using the general Stiefel manifold $V_k(\R^n)$ it is possible to generalize Corollary~\ref{corollary:outscr-paral} to $(n-k)$-cylinders over $k$-dimensional parallelotopes.
An $(n-k)$-cylinder is a convex regions in $\R^n$ bounded by $2k$ pairwise distinct hyperplanes $H_1^0, H_1^1,\ldots, H_k^0, H_k^1$ that are parallel in pairs, $H_i^0 \parallel H_i^1$ ,
and $H_1^0,\ldots,H_n^0$ are in general position.
Each such convex region has a $k$-dimensional cross-section volume.
We see that among the $(n-k)$-cylinders over $k$-dimensional parallelotopes outscribed around the convex body $K\subset \R^n$ there are at least $nk - \frac{k(k+1)}{2} + 1$ critical ones
with respect to the cross-section volume.
This ``criticality'' may be described in terms of segments between opposite support points, like in Definition~\ref{definition:critical}.
\end{remark}

\subsection{Birkhoff--James orthogonal bases}
\label{subsection:Birkhoff--James}

Similarly, as in Corollary~\ref{corollary:outscr-paral}, we estimate the number of Birkhoff--James orthonormal bases in a finite dimensional normed space.
Assume that the norm $\|\cdot\|$ in $\R^n$ is smooth.

\begin{definition}
The vectors $x\in\R^n$ and $y\in\R^n$ are \emph{Birkhoff--James orthogonal} with respect to $\|\cdot \|$, and denoted by $x\perp_{BJ} y$, if $\tfrac{d}{dt} \| x + ty\| _{t=0} = 0$.
\end{definition}

This relation needs not to be symmetric, that is $x\perp_{BJ} y$ is not equivalent to $y\perp_{BJ} x$.
In fact, for $n\ge 3$, the Birkhoff--James orthogonality is symmetric if and only if the norm is Euclidean.
When considering a Birkhoff--James orthogonal basis $(e_1,\ldots, e_n)$ of $\R^n$ it is important to be aware that $e_i \perp_{BJ} e_j$ and $e_j \perp_{BJ} e_i$ are two different conditions.

\begin{corollary}
\label{corollary:bj-bases}
Every smooth norm on  $\R^n$ has at least $\tfrac{n(n-1)}{2}+1$ Birkhoff--James orthonormal bases.
\end{corollary}

\begin{proof}
Consider again the manifold $M$ of all bases $(e_1, \ldots, e_n)$ of $\R^n $such that $\|e_1\| = \dots = \|e_n\| = 1$.
The function $f\colon M\to\R$ defined by
\[
f(e_1, \ldots, e_n): = \tfrac{1}{|\det (e_1, \ldots, e_n)|}
\]
is proper, smooth, and $W_n$-invariant on $M$.
Hence it has at least $\cat M/W_n = \frac{n(n-1)}{2}+1$ critical points.

Now, for a pair of distinct indexes $(i,j)$ and the fixed basis $(e_1,\ldots, e_n)$ we consider the function $\phi_{(i,j)}\colon\R\to\R$ defined by
\[
\phi_{(i,j)}(t): = \tfrac{1}{\det \left(e_1, \ldots, \frac{e_i + t e_j}{\|e_i + t e_j\|}, \ldots, e_n\right)}.
\]
The first derivative of this function evaluated at $t=0$ is
\[
\phi_{(i,j)}'(0)
=
\Big(\tfrac{\|e_i + t e_j\|}{\det (e_1, \ldots, e_i + t e_j, \ldots, e_n)}\Big)'_{t=0}
=
\Big(\tfrac{\|e_i + t e_j\|}{\det (e_1, \ldots, e_i, \ldots, e_n)}\Big)'_{t=0}
=
\tfrac{\|e_i + t e_j\|'_{t=0}}{\const}.
\]

Let $E_{ij}(t)\in \mathrm{GL}(n)$ be the elementary matrix with $t$ at position $(i,j)$, units on the diagonal, and zeroes elsewhere.
Then $\phi'_{(i, j)}(0)=\tfrac{d}{dt}f((e_1,\ldots, e_n)\cdot E_{ij}(t))$, where ``$\cdot$'' denotes the right action defined in Section \ref{subsection:2.2}.
For a critical point of $f$, all such derivatives must vanish and thus a critical point of the function $f$ is also a Birkhoff--James orthonormal basis.
This concludes the proof.
\end{proof}

\end{document}